\documentclass[a4paper,reqno]{amsart}
\usepackage{amssymb, amsmath, amsthm, amscd}
\usepackage{bbold}
\usepackage[colorlinks, linkcolor=blue, citecolor=red]{hyperref}
\allowdisplaybreaks
\newcommand{\mathscr}[1]{\mathcal{#1}}
\newcommand{\mathbbm}[1]{\mathbb{#1}}

\theoremstyle{plain}
\newtheorem{theorem}{Theorem}[section]

\newtheorem{corollary}[theorem]{Corollary}
\newtheorem{lemma}[theorem]{Lemma}
\newtheorem{proposition}[theorem]{Proposition}

\numberwithin{equation}{section}

\def\E{{\mathbb E}}

\def\R{{\mathbb R}}

\def\P{{\mathbb P}}
\newcommand{\ds}{\displaystyle}
\newcommand\numberthis{\addtocounter{equation}{1}\tag{\theequation}}

\author{M. Baroun}
\address{Cadi Ayyad University, Faculty of Sciences Semlalia, 2390, Marrakesh, Morocco}
\address{Laboratory of Mathematics, Modeling and Automatic Systems, Marrakesh, Morocco}
\email{m.baroun@uca.ac.ma}

\author{Mohamed Fadili}
\thanks{Corresponding author: Mohamed Fadili (m.fadili@uca.ma).}
\address{Ecole Normale Sup\'erieure, Universit\'e Cadi Ayyad, Marrakech, Morocco}
\address{Laboratory of Mathematics, Modeling and Automatic Systems, Marrakech, Morocco}
\email{m.fadili@uca.ac.ma}

\author{A. Khchine}
\address{Cadi Ayyad University, National School of Applied Sciences, 575, Marrakesh, Morocco}
\address{Laboratory of modelisation  of complex systems (LMCS), Marrakesh, Morocco}
\email{a.khchine@uca.ac.ma}

\author{L. Maniar}
\address{Lahcen Maniar, Cadi Ayyad University, Faculty of Sciences Semlalia, 2390, Marrakesh, Morocco}
\address{Laboratory of Mathematics, Modeling and Automatic Systems, Marrakesh, Morocco}	
\email{maniar@uca.ma}

\title[]{Carleman Estimates and Controllability of Stochastic Degenerate Parabolic Heat Equations}


\begin{document}
\begin{abstract}
We study null controllability for a class of stochastic degenerate parabolic equations in both the weakly and strongly degenerate regimes. We first prove a global Carleman estimate for a linear forward stochastic degenerate equation with multiplicative noise. This estimate yields an observability inequality and a unique continuation property, which in turn imply null controllability for a backward stochastic degenerate equation. We then establish a second Carleman estimate for a backward equation with a weight that does not vanish at $t=0$ and combine it with the Hilbert Uniqueness Method (HUM). As a result, we obtain null controllability for a forward stochastic degenerate equation with two controls, together with weighted estimates for the controls and the state.
\end{abstract}
\maketitle
\enlargethispage{3pt}

\section{Introduction and main results}

Degenerate parabolic equations describe diffusion processes whose smoothing effect
weakens or vanishes at part of the boundary. The resulting loss of uniform
ellipticity changes both the natural energy space and the boundary condition at
the degeneracy point. When the equation is driven by noise, controllability
requires estimates that account simultaneously for this singular spatial
structure and for the adaptedness of the state and the controls. In particular,
the deterministic Carleman framework cannot be transferred directly, since
It\^o correction terms and the martingale component must enter the weighted
estimates.

This paper develops such estimates for one-dimensional stochastic parabolic
equations whose diffusion coefficient may be weakly or strongly degenerate at
$x=0$. The analysis has two purposes. The first is to derive an observability
inequality for a forward equation with multiplicative noise and to use it to
control a backward stochastic degenerate equation. The second is to establish a
Carleman estimate adapted to the initial time and, through the Hilbert
Uniqueness Method, to obtain null controllability for a forward equation with a
control in both the drift and the diffusion.

We first consider the backward stochastic degenerate equation
\begin{equation*}\label{BSHE}\tag{BSDE}
\begin{cases}
\ds d_t z+\left(a(x) z_x\right)_x dt +b z(t,x)dt +c k dt= \mathbbm{1}_{\mathscr{O}}v dt +kdW(t) \quad\text{ in } Q_T,\\
C z=0,\qquad \text{ on } \Sigma_T ,\\
z(T,\cdot)=z_T, \qquad \text{ in } (0,1),
\end{cases}
\end{equation*}	
where the solution pair $(z,k)$ is understood in the sense of
\cite{hupeng,Parpeng}, $\mathscr{O}$ is a nonempty open subset of $(0,1)$,
$\mathbbm{1}_{\mathscr{O}}$ is its characteristic function, and $v$ is a
distributed control supported in $\mathscr{O}$.

We also consider the forward stochastic degenerate equation
\begin{equation*}\label{form6.1}\tag{FSDE2}
\begin{cases}
\ds	dy=\left[ \left(a(x) y_x\right)_x  +F+ \mathbbm{1}_{\mathscr{O}} h\right]dt + \left(G+H\right) dW(t),\quad\text{ in } Q_T,\\
C y=0,\qquad \text{ on } \Sigma_T,\\
y(0,\cdot)=y_0(\cdot)\qquad \text{ in } (0,1),
\end{cases}
\end{equation*} 
where the pair $(h,H)$ acts respectively on the drift and the diffusion.

Let $T >0$, $Q_T=(0,T)\times(0,1)$, $\Sigma_T=(0,T)\times\{0,1\}$, $\mathscr{O}\subset (0,1)$, and $\mathscr{O}_T=(0,T)\times\mathscr{O}$. Let $(\Omega, \mathcal{F}, {\mathcal{F}_t}_{t\geq0}, \P)$ be a complete filtered probability space carrying a one-dimensional standard Brownian motion ${W(t)}_{t\geq0}$, and let ${\mathcal{F}_t}_{t\ge0}$ be its natural filtration augmented by all $\P$-null sets. The diffusion coefficient $a$ may degenerate at $x=0$ (i.e., $a(0)=0$) and satisfies one of the following assumptions:

$i)$ The weak degeneracy:
\begin{equation}\label{WD}
\text{(WD)},\quad	
\begin{cases}
(i)\; a \in \mathcal{C}([0,1])\cap\mathcal{C}^1((0,1]),\; a >0 \text{ in } (0,1],\;a(0)=0,\\
(ii)\;\exists K\in[0,1)\text{ such that }xa'(x)\leqslant Ka(x),\quad\text{for all } x\in[0,1],
\end{cases}
\end{equation}
or,\
$ii)$ The strong degeneracy:
\begin{equation}\label{SD}
\text{(SD)},\quad
\begin{cases}
(i)\; a \in \mathcal{C}^1([0,1]),\; a >0 \text{ in } (0,1],\;a(0)=0,\\
(ii)\;\exists K\in[1,2)\text{ such that }xa'(x)\leqslant Ka(x),\text{ for all } x\in[0,1], \\
(iii)\;\begin{cases}
\ds\exists\theta\in(1,K],\; x\mapsto\frac{a(x)}{x^{\theta}}\text{ is nondecreasing near } 0, \text{ if } K >1,\\
\ds\exists\theta\in(0,1),\; x\mapsto\frac{a(x)}{x^{\theta}}\text{ is nondecreasing near } 0, \text{ if } K=1.
\end{cases}
\end{cases}
\end{equation}
In both regimes, $x\mapsto x^2/a(x)$ is nondecreasing on $(0,1]$. Indeed,
\[
\left(\frac{x^2}{a(x)}\right)'=
\frac{x\bigl(2a(x)-xa'(x)\bigr)}{a(x)^2}
\geqslant (2-K)\frac{x}{a(x)}>0,
\qquad x\in(0,1].
\]
Thus this monotonicity is a consequence of \eqref{WD}--\eqref{SD}, rather than
an additional assumption. The
coefficients $b$, $c$, $F$, and $G$ are subject to the assumptions stated
below. The boundary operator $C$ imposes Dirichlet conditions in the weakly
degenerate case and, in the strongly degenerate case, $u(t,1)=0$ together with
the natural weighted Neumann condition
$\lim_{x\to0^+}a(x)u_x(t,x)=0$.

The controllability of stochastic parabolic equations was initiated, in this
setting, by the Carleman approach of Barbu, R\u{a}\c{s}canu, and Tessitore
\cite{barbuRascanu}. Tang and Zhang \cite{tangzhang} subsequently obtained null
controllability results for forward and backward stochastic parabolic
equations. The method has since been extended to semilinear equations and
systems; see, among others, \cite{Santamaria,Fadi1}. These works also show a
feature specific to forward stochastic equations: in general, a drift control
alone does not compensate for the martingale component, and a second control in
the diffusion is needed. The duality argument used below is based on the
Hilbert Uniqueness Method introduced by Lions \cite{J.Lions1}.

For degenerate stochastic equations, Liu and Yu \cite{liuyu} proved Carleman
estimates in the model case $a(x)=x^\alpha$, $\alpha\in(0,2)$, and derived
controllability consequences. The present work treats a broader class of
diffusion coefficients characterized by \eqref{WD} and \eqref{SD}. Its main
contributions are the following. First, a single spatial weight is constructed
for both weak and strong degeneracy, leading to a global Carleman estimate for
the forward equation below. Second, the resulting observability inequality
yields null controllability of the backward equation \eqref{BSHE}. Third, a
modified time weight, bounded at $t=0$, gives a weighted HUM construction for
\eqref{form6.1}, including estimates for the state and both controls.

The first Carleman estimate is established for
\begin{equation*}\label{equ:3.16}\tag{FSDE}
\begin{cases}
\ds d_t y-\left[\left(a(x) y_x\right)_x +b y(t,x)\right]dt= c ydW(t) ,\quad\text{ in } Q_T,\\
C y=0 ,\quad \text{ on } \Sigma_T ,\\
y(0,\cdot)=y_0(\cdot) ,\quad \text{ in } (0,1),
\end{cases}
\end{equation*}
where $a$ satisfies either \eqref{WD} or \eqref{SD}; see Theorem
\ref{theorem3.10}. The associated observability inequality also implies a
unique continuation property. By duality, it gives the first main result.

\begin{theorem}[Null controllability of the backward stochastic degenerate equation \eqref{BSHE}]\label{meantheorem}~\par
For every final datum $z_T$ in $L^2(\Omega,\mathcal{F}_T;L^2(0,1))$, there exists a control $v$ in $L_{\mathcal{F}}^2\left(\Omega;L^2(0,T;L^2(0,1))\right)$ such that $z(0,x)=0$ for almost every $x\in (0,1)$.
\end{theorem}

For the forward system \eqref{form6.1}, we prove a second global Carleman
estimate whose time weight remains bounded at $t=0$. This estimate supplies the
weighted coercivity needed in the HUM argument and leads to the following
result.

\begin{theorem}\label{meantheorem2}
Fix $s>s_0$, where $s_0$ is the threshold in Proposition \ref{Prop5.1}.
For any $y_0\in L^2(\Omega,\mathcal{F}_0;L^2((0,1),a(x)x^{-2}\,dx))$ and any pair $(F,G)\in\mathcal{S}_s$, there exists a pair of controls
$\ds (\widehat{h},\widehat{H})\in L_{\mathcal{F}}^2(0,T;L^2(\mathscr{O}))\times L_{\mathcal{F}}^2(0,T;L^2(0,1))$ such that the associated solution $\ds \widehat{y}$ to the system \eqref{form6.1} satisfies $\ds \widehat{y}(T)=0$ in $(0,1)$ a.s. Moreover, we have the following estimate:
\begin{multline}\label{equtheo5.2}
\ds \E\int_{Q_T}e^{2s \overline{\varphi}}\widehat{y}^2 \,dx\,dt
+\E\int_0^T\int_{\mathscr{O}}e^{2s \overline{\varphi}}s^{-3}\bar{\theta}^{-3}\widehat{h}^2 \,dx\,dt \\
+\E\int_{Q_T}e^{2s\overline{\varphi}}s^{-2}\bar{\theta}^{-2}\widehat{H}^2 \,dx\,dt
\leqslant C \E\int_0^1\frac{a(x)}{x^2}|y_0(x)|^2\,dx
+ C \|(F,G)\|_{\mathcal{S}_s}^2 .
\end{multline}
\end{theorem}
For each fixed $s>s_0$, the constant in \eqref{equtheo5.2} may depend on
$s$ (as well as on the structural data); in particular, it includes the
constant $C_s$ introduced in \eqref{initial-weight-bound} below.
Here $\mathcal{S}_s$ is the space
{\small
\begin{multline*}
\mathcal{S}_s=\Bigl\{(F,G)\in L_{\mathcal{F}}^2(0,T;L^2(0,1))^2 : \\
\E\int_{Q_T} e^{2s\overline{\varphi}} s^{-3}\bar{\theta}^{-3}\frac{a}{x^2}F^2\,dx\,dt
+\E\int_{Q_T} e^{2s\overline{\varphi}} s^{-2}\bar{\theta}^{-2}G^2\,dx\,dt <\infty\Bigr\}.
\end{multline*}
}
endowed with its canonical norm. The functions $\overline{\varphi}$ and
$\bar{\theta}$ are defined in \eqref{newtheta}.

Section 2 introduces the weighted functional setting and recalls the
well-posedness results used throughout the paper. Section 3 proves the forward
Carleman estimate. Section 4 derives observability, unique continuation, and
Theorem \ref{meantheorem}. Section 5 establishes the modified Carleman estimate
and proves Theorem \ref{meantheorem2}. The appendix contains a stochastic
Caccioppoli inequality.

Throughout the paper, $C$ denotes a generic positive constant that may change from line to line.

\section{Preliminaries}
We introduce the weighted spaces associated with the diffusion coefficient $a$.
We shall also use the Hilbert space
\[
L^2\!\left((0,1),\frac{a(x)}{x^2}\,dx\right)
:=\left\{u:(0,1)\to\mathbb{R}\text{ measurable}:\
\int_0^1\frac{a(x)}{x^2}|u(x)|^2\,dx<\infty\right\},
\]
endowed with the inner product
$\langle u,v\rangle_{a/x^2}=\int_0^1\frac{a(x)}{x^2}u(x)v(x)\,dx$.
Its completeness follows from the standard completeness of $L^2$ spaces
with respect to the measure $\frac{a(x)}{x^2}\,dx$.
In the (WD) case:
$$\displaystyle H_{a}^1=\left\lbrace u \in L^2(0,1) : u \text{ absolutely continuous in } [0,1], \sqrt{a}u_x \in L^2(0,1) \text{ and } u(1)=u(0)=0 \right\rbrace$$
and
$$\displaystyle H_{a}^2=\left\lbrace u \in H_{a}^1(0,1) : au_x \in H^1(0,1) \right\rbrace.$$
In the (SD) case:
$$\displaystyle H_{a}^1=\left\lbrace u \in L^2(0,1) : u \text{ absolutely continuous in } (0,1], \sqrt{a}u_x \in L^2(0,1) \text{ and } u(1)=0 \right\rbrace$$
and
\begin{align*}
	\displaystyle H_{a}^2  &=\left\lbrace u \in H_{a}^1(0,1) : au_x \in H^1(0,1) \right\rbrace\\
	 &=\big\lbrace u \in L^2(0,1) : u \text{ absolutely continuous in } (0,1],\,
	au \in H_0^1(0,1),\, au_x \in H^1(0,1)\\
	 &\hspace{5.5cm}\text{and } (au_x)(0)=0 \big\rbrace.
\end{align*}
In both cases, the norms are defined as follows:
\begin{equation}\label{norms}
	\|u\|_{H_{a}^1}^2  =\|u\|_{L^2(0,1)}^2+\|\sqrt{a} u_x\|_{L^2(0,1)}^2, \,\,\,\|u\|_{H_{a}^2}^2 =\|u\|_{H_{a}^1}^2+\|{(a u_x )}_x\|_{L^2(0,1)}^2.
\end{equation}
In the sequel, equations \eqref{equ:3.16} and \eqref{BSHE} are studied under the following assumptions:
\begin{enumerate}
\item[1)] $\mathscr{O} \subset (0,1)$ is an open non-empty subset of the interval $(0,1)$ such that $\overline{\mathscr{O}}\subset (0,1)$.
\item[2)] $\{W(t); t\geq 0\}$ is a standard one-dimensional Brownian motion on a complete probability space $(\Omega,\mathcal{F},\P)$ endowed with the filtration $\mathcal{F}_t=\sigma\{W(s):s\in [0,t]\}\vee \{A\in \mathcal{F}, \P(A)=0\}$.
\item[3)] The coefficients $b, c$ are adapted processes with values in $W^{1,\infty}(0,1)$ and $b,c \in L^{\infty}(\Omega\times[0,T];W^{1,\infty}(0,1))$.
\end{enumerate} 

For a Banach space $X$, we denote by $L_{\mathcal{F}}^2(\Omega;L^2(0,T;X))$ the space of all $X$-valued $\{\mathcal{F}_t\}_{t\geq0}$-adapted processes $z$ such that $ \E\big( \|z\|^2_{L^2(0,T;X)}\big)  < \infty$, endowed with its canonical norm. We similarly denote by $L_{\mathcal{F}}^2(\Omega;\mathcal{C}([0,T];X))$ the space of all $X$-valued adapted continuous processes $z$ such that $\E\big(\|z\|^2_{\mathcal{C}([0,T];X)}\big) <\infty$, and by $L^{\infty}_{\mathcal{F}}(\Omega;X)$ the space of all adapted essentially bounded $X$-valued processes.

Consider the forward stochastic equation:
\begin{equation}\label{FSHE1}
	\begin{cases}
		\ds d_t y-\left[\left(a(x) y_x\right)_x +b_{0} y(t,x)\right]dt= f(t,x)dt +c_{0} ydW(t) \quad\text{ in } Q_T,\\
		C y=0 \qquad \text{ on } \Sigma_T ,\\
		y(0,\cdot)=y_0(\cdot) \qquad \text{ in } (0,1),  
	\end{cases}
\end{equation}
where 
\begin{equation}
	\begin{cases}
		i) \, f \in L_{\mathcal{F}}^2(\Omega; L^2(0,T;L^2(0,1))),\\
		ii)\, b_{0}, c_{0} \in L_{\mathcal{F}}^{\infty}(\Omega;L^{\infty}(0,T;L^{\infty}(0,1))),\\
		iii)\,  y_0 \in  L^{2}(\Omega,\mathcal{F}_0, \P;L^2(0,1)). 
	\end{cases}
\end{equation}
Under these assumptions, it is well known (see \cite{barbuRascanu,PratoZab}) that there exists a unique solution $y$ of \eqref{FSHE1} belonging to the space $L_{\mathcal{F}}^2(\Omega; \mathcal{C}([0,T];L^2(0,1)))\cap L_{\mathcal{F}}^2(\Omega; L^2(0,T;H_a^1))$, 
with for a suitable constant $C$ independent of $y_0$ and $f$:
$$\E\left[\sup_{t\in[0,T]}\|y(t)\|_{L^2(0,1)}^2+\int_{0}^{T}\|y(t)\|_{H_a^1}^2 dt\right]\leqslant C \left(\|y_0\|_{L^2(0,1)}^2+\E \int_{0}^{T}\|f\|_{L^2(0,1)}^2 dt\right).$$
Furthermore, if $y_0\in H_a^1$, then $y\in L_{\mathcal{F}}^2(\Omega; \mathcal{C}([0,T];H_a^1))\cap L_{\mathcal{F}}^2(\Omega; L^2(0,T;H_a^2))$.

We also consider the backward equation:
\begin{equation}\label{BSHE1}
	\begin{cases}
		\ds d_t z+\left(a(x) z_x\right)_x dt +b_0 z(t,x)dt +c_0 k dt= G(t,x) dt +kdW(t) \quad\text{ in } Q_T,\\
		C z=0\qquad \text{ on } \Sigma_T ,\\
		z( T,\cdot)=z_T \qquad \text{ in } (0,1),  
	\end{cases}
\end{equation}
where 
\begin{equation}\label{BSHEassum}
	\begin{cases}
		i) \, G \in L_{\mathcal{F}}^{2}(\Omega; L^2(0,T;L^2(0,1))),\\
		ii)\, z_T\in L^{2}(\Omega,\mathcal{F}_T, \P;L^2(0,1)).
	\end{cases}
\end{equation} 
Similarly, by \cite{ParRasca,Tess}, under \eqref{BSHEassum} the backward equation \eqref{BSHE1} has a unique solution
$(z,k)$ with $z \in L_{\mathcal{F}}^{2}(\Omega; \mathcal{C}([0,T];L^2(0,1)))\cap L_{\mathcal{F}}^{2}(\Omega; L^2(0,T;H_a^1))$ 
and $k\in L_{\mathcal{F}}^{2}(\Omega; L^2(0,T;L^2(0,1)))$.

\section{Carleman estimate for forward stochastic parabolic degenerate equation}
We now consider the following uncontrolled stochastic problem, for which the diffusion coefficient $a$ degenerates at zero, that is, $a(0)=0$ and $a >0$ in $(0,1]$:
\begin{equation}\label{nequa3.1}
	\begin{cases}
		d_ty(t,x)=\left(a(x) y_x\right)_xdt +f(t,x)dt+g(t,x)dW(t),\,\,\text{ in } Q_T,\\
		C y=0,\qquad \text{ on } \Sigma_T,\\
		y(0,\cdot)=y_0(\cdot), 
	\end{cases}
\end{equation}
where $f$ belongs to $L_{\mathcal{F}}^2(\Omega,L^2(0,T;L^2(0,1)))$ and $g$ belongs to $L_{\mathcal{F}}^2(\Omega,L^2(0,T; H_a^1))$.

To define the weight function, fix a nonempty open subset $\mathscr{O}=(a,b) \Subset(0,1)$ and $\mathscr{O}_1=(a_1,b_1)\Subset \mathscr{O}=(a,b)$. Let $\xi \in \mathcal{C}^3(\R)$ satisfy $0\leq \xi\leq 1$ and
$$\xi(x)=\begin{cases}
	1\quad \text{ if } x\in [0,a_1],\\ 0\quad \text{ if } x\in [b_1,1]. 
\end{cases}$$
Define $\ds \phi(x)=d-\int_{0}^{x}\frac{r}{a(r)}dr$, where the real $d$ is chosen such that $\phi >0$
on $(0,1)$, and let $\rho$ be defined by $\ds \rho(x)=\int_{x}^{1}\frac{r}{a(r)}dr$, for all $x\in (0,1)$. Fix $\lambda>0$ sufficiently large and define
$\ds \psi(x)=e^{2\lambda\|\rho\|_{\infty}}-e^{\lambda\rho(x)}$ and 
$$\ds \beta(x)=\xi(x)\phi(x)+(1-\xi(x))\psi(x).$$
Notice that $\beta\geq0$ by construction. Finally, we set $$\ds \varphi(t,x)=\theta(t)\beta(x),$$
where $\ds\theta(t)=\frac{1}{t^4(T-t)^4}$ on $(0,T)$.

We recall some useful properties of the function $\theta$ (see \cite{hjjaj}).
\begin{lemma}\label{rem3.5}
	The function $\theta$ satisfies
	$$\lim\limits_{t\to 0^+}\theta(t)=\lim\limits_{t\to T^-}\theta(t)=+\infty,\quad \theta(t)\geq c_1, \quad |\dot{\theta}|\leq c_2 \theta^2,\quad |\ddot{\theta}(t)|\leq c_3 \theta^3,$$
	where $c_1=(\frac2T)^8,\,\, c_2=8(\frac{T}{2})^7$ and $ c_3=80(\frac{T}{2})^{14}$.\\
	Moreover, we have $\ds |\dot{\theta}(t)|\leq c_4 \theta^{\frac32}$ and $\ds |\ddot{\theta}(t)|\leq c_5 \theta^2$ with $c_4=T^3$ and $c_5=80(\frac{T}{2})^6$.
\end{lemma}

Recall that the unbounded operator $\ds \mathcal{M}\,:\,\mathcal{D}(\mathcal{M})\subset L^2(0,1)\longrightarrow L^2(0,1)$ defined by $$\mathcal{M}y=(a(x)y_x)_x,\quad \mathcal{D}(\mathcal{M})=H_a^2$$
generates a contraction strongly continuous semi-group $(T(t))_{t\geq0}$ (see \cite{camp}). The following result is a weighted identity for the forward stochastic degenerate parabolic operator $dy-\mathcal{M}ydt$. Its proof may be found in \cite{liuyu}.
\begin{lemma}\label{lem3.2}
	Let $y$ be a $H_a^2$-valued continuous semimartingale, and set $z=e^{-s\varphi}y$. Then, for a.e.\,\,$(t,x)\in Q_T$ and $\P$-a.s. $\omega\in \Omega$, one has the following weighted identity:
	\begin{align*}
		 &e^{-s\varphi}\left[A z-(az_x)_x\right]\left[dy-(ay_x)_x dt\right]\\
		 &=\left[A z-(az_x)_x\right]^2 dt +d\left(\frac12 A z^2+\frac12az_x^2\right)-\frac12A(dz)^2-\frac12 a (dz_x)^2-(az_xdz)_x\\
		 &\qquad+\left[sa(a\varphi_x)_xz_x z-sAa\varphi_x z^2-\frac12 s a(a\varphi_x)_{x x}z^2+sa^2\varphi_x z_x^2\right]_x dt \numberthis\label{nequa3.2}\\
		 &\qquad\quad+\left\{s(A a \varphi_x)_x-sA(a\varphi_x)_x+\frac12[sa(a\varphi_x)_{x x}]_x-\frac12A_t\right\}z^2 dt\\
		 &\qquad\qquad-\left\{s a[(a\varphi_x)_x+a\varphi_{x x}]\right\}z_x^2 dt,
	\end{align*}
	where $\ds A=s\varphi_t-s^2a\varphi_x^2$.
\end{lemma}

Let $y$ be any solution of \eqref{nequa3.1} and notice that $e^{-s\varphi}(0,x)=e^{-s\varphi}(T,x)=0$ in $[0,1]$. Integrating \eqref{nequa3.2} on $Q_T$ and taking expectation, one obtains that:
\begin{align*}
	 &\E\int_{Q_T} e^{-s\varphi}\left[A z-(az_x)_x\right]\left[dy-(ay_x)_x dt\right]dx\\
	 &=\E\int_{Q_T}\left[A z-(az_x)_x\right]^2 dxdt -\frac12\E\int_{Q_T}Ae^{-2s\varphi}g^2dx dt -\frac12\E\int_{Q_T} a [e^{-s\varphi}g]_x^2 dxdt\\
	 &+\E\int_{Q_T}\left[sa(a\varphi_x)_xz_x z-sAa\varphi_x z^2-\frac12 s a(a\varphi_x)_{x x}z^2+sa^2\varphi_x z_x^2\right]_x dxdt\\
	 & -\E\int_{Q_T}(az_xdz)_xdxdt \numberthis\label{nequa3.3}\\
	 &+\E\int_{Q_T}\left\{s(A a \varphi_x)_x-sA(a\varphi_x)_x+\frac12[sa(a\varphi_x)_{x x}]_x-\frac12A_t\right\}z^2 dxdt\\
	 &-\E\int_{Q_T}\left\{s a[(a\varphi_x)_x+a\varphi_{x x}]\right\}z_x^2 dxdt.
\end{align*}
To estimate terms on the right of the expression \eqref{nequa3.3}, let us write:
\begin{align*}
	 &Q_1= -\frac12\E\int_{Q_T}Ae^{-2s\varphi}g^2dx dt,\quad Q_2= -\frac12\E\int_{Q_T} a [e^{-s\varphi}g]_x^2dx dt,\quad Q_3=-\E\int_{Q_T}(az_xdz)_xdxdt,\\
	 &Q_4=\E\int_{Q_T}\left[sa(a\varphi_x)_xz_x z-sAa\varphi_x z^2-\frac12 s a(a\varphi_x)_{x x}z^2+sa^2\varphi_x z_x^2\right]_x dx dt,\\
	 &Q_5=\E\int_{Q_T}\left\{s(A a \varphi_x)_x-sA(a\varphi_x)_x+\frac12[sa(a\varphi_x)_{x x}]_x-\frac12A_t\right\}z^2 dxdt,\\
	 &Q_6=-\E\int_{Q_T}\left\{s a[(a\varphi_x)_x+a\varphi_{x x}]\right\}z_x^2 dxdt.
\end{align*}
This leads to the following lemmas:
\begin{lemma}\label{lem3.3}
	$$Q_1\geqslant -C\E\int_{Q_T} s^2\theta^2 e^{-2s\varphi}g^2dx dt.$$
\end{lemma}
\begin{proof}
	Indeed,
	$\ds A=s\varphi_t-s^2a\varphi_x^2=s\dot{\theta}\beta-s^2a\theta^2\beta_x^2\leqslant C  s^2\theta^2$.	 
\end{proof}

\begin{lemma}\label{lem3.4}
	$$Q_2= -\frac12\E\int_{Q_T} a [e^{-s\varphi}g]_x^2 dx dt\geqslant -C\E\int_{Q_T} \left( s^2 \theta^2\frac{x^2}{a} e^{-2s\varphi}g^2 +a e^{-2s\varphi}g_x^2 \right) dx dt.$$	
\end{lemma}
\begin{proof}
	Since $[e^{-s\varphi}g]_x=-s\varphi_x e^{-s\varphi}g +e^{-s\varphi}g_x$, then 
	$$\ds Q_2= -\frac12\E \int_{Q_T} a [e^{-s\varphi}g]_x^2 dx dt\geqslant -C\E\int_{Q_T} \left( a s^2\varphi_x^2 e^{-2s\varphi}g^2 +a e^{-2s\varphi}g_x^2 \right)dx dt.$$
	On the interval $[0,a_1]$ we have $\ds a\varphi_x^2=\theta^2\frac{x^2}{a}$, then we can bound $\ds a\varphi_x^2$ over $[0,1]$ by $\ds C\theta^2\frac{x^2}{a}$.	
\end{proof}

\begin{lemma}\label{lem3.5}
	$$Q_3=0.$$	
\end{lemma}
\begin{proof}
	For smooth $H_a^2$-valued semimartingales,
	\[
	Q_3=-\E\int_0^T\left[az_x\,dz\right]_{x=0}^{x=1}.
	\]
	A Dirichlet trace that vanishes for every $t$ is the identically zero
	semimartingale; hence its stochastic differential also vanishes. In the
	weakly degenerate case, $z(t,0)=z(t,1)=0$ therefore gives
	$dz(t,0)=dz(t,1)=0$, and both boundary contributions vanish.

	In the strongly degenerate case, the same argument gives $dz(t,1)=0$.
	At the degenerate endpoint, using $z=e^{-s\varphi}y$ and
	$\varphi_x=-\theta x/a$ on $[0,a_1]$, we have
	\[
	az_x=s\theta x e^{-s\varphi}y+e^{-s\varphi}ay_x.
	\]
	Since $(ay_x)(t,0)=0$, both terms on the right vanish at $x=0$.
	Thus the contribution at $x=0$ vanishes as well, and $Q_3=0$ in both
	regimes. The identity for solutions with the regularity used below follows
	by the same approximation argument as for the weighted identity
	\eqref{nequa3.3}.
\end{proof}

\begin{lemma}\label{lem3.6}
	$$Q_4\geqslant 0.$$	
\end{lemma}
\begin{proof}
	We notice that $\ds \varphi_x=\lambda\theta\frac{x}{a}e^{\lambda\rho(x)}$ for all $x$ in $[b_1,1]$ and $\ds \varphi_x=-\theta\frac{x}{a},\,\, (a\varphi_x)_{x x}=0$ for all $x$ in $[0,a_1]$. Then, in the weak degenerate case (WD) we have:
	\begin{align*}
		Q_4 &=\E\int_0^T\left[sa(a\varphi_x)_xz_x z-sAa\varphi_x z^2-\frac12 s a(a\varphi_x)_{x x}z^2+sa^2\varphi_x z_x^2\right]_{x=0}^{x=1} dt\\
		 &=\E\int_0^T\left[sa^2\varphi_x z_x^2\right]_{x=0}^{x=1} dt=\E\int_0^T s\lambda xa\theta e^{\lambda\rho(x)} z_x^2\big|_{x=1}dt \geqslant 0.	 
	\end{align*}
	Whereas in the strong degenerate case (SD), we get:
	\begin{align*}
		Q_4 &=\E\int_0^Tsa^2\varphi_x z_x^2\big|_{x=1} dt-
		\E\int_0^T -sAa\varphi_x z^2-\frac12 s a(a\varphi_x)_{x x}z^2+sa^2\varphi_x z_x^2\big|_{x=0} dt\\
		 &=\E\int_0^Ts\lambda a x\theta e^{\lambda\rho(x)} z_x^2\big|_{x=1} dt-
		\E\int_0^T sA\theta x  z^2-sa \theta x  z_x^2\big|_{x=0} dt.
	\end{align*}
	For smooth solutions, the contribution at $x=0$ vanishes because $x=0$ and
	$\ds x^2/a(x)\to0$ as $x\to0^+$ under assumptions (WD) and (SD); the general
	case follows by density. Therefore,
	\[
		Q_4=\E\int_0^Ts\lambda a x\theta e^{\lambda\rho(x)}
		z_x^2\big|_{x=1}dt\geqslant0.
	\]
\end{proof}

\begin{lemma}\label{lem3.7}
	There exist two positive constants $C^{\prime}$ and $C^{\prime\prime}$ such that for every choice of small scalar $\varepsilon >0$ we have:
	\begin{equation}\label{equlem3.7}
		Q_5\geqslant C^{\prime}\E\int_{Q_T} s^3 \theta^3 \frac{x^2}{a}z^2 dxdt 
		-C^{\prime\prime}\E\int_{0}^{T}\int_{[a_1,b_1]} s^3 \theta^3 \frac{x^2}{a}z^2 dxdt-\varepsilon\E\int_{Q_T} (s\theta a z_x^2+  s^{3}\theta^3 \frac{x^2}{a}  z^2) dxdt.
	\end{equation}
\end{lemma}
\begin{proof}
   Let $\varepsilon >0$ be fixed. Since $\ds A_t=s\varphi_{t t}-2s^2a\varphi_x\varphi_{x t}$ and $\ds A_x=s\varphi_{x t}-s^2(a\varphi_x^2)_x$, then:
	\begin{align*}
		 &s(A a \varphi_x)_x-sA(a\varphi_x)_x+\frac12[sa(a\varphi_x)_{x x}]_x-\frac12A_t \\
		 &= sA_x a \varphi_x+\frac12 sa_x(a\varphi_x)_{x x}+\frac12sa(a\varphi_x)_{x x x}-\frac12A_t\\
		 &=\frac12 sa_x(a\varphi_x)_{x x}+\frac12sa(a\varphi_x)_{x x x}+2s^2a\varphi_x\varphi_{x t}- s^3(a\varphi_x^2)_x a \varphi_x-\frac12 s\varphi_{t t}\\
		 &=\frac12 sa_x(a\varphi_x)_{x x}+\frac12sa(a\varphi_x)_{x x x}+2s^2a\varphi_x\varphi_{x t}- s^3 a \varphi_x^2 ((a\varphi_x)_x+a\varphi_{x x} )-\frac12 s\varphi_{t t}.
	\end{align*} 
	Hence, 
	\begin{align*}
		Q_5 &=\E\int_{Q_T}\left(\frac12 sa_x(a\varphi_x)_{x x}+\frac12sa(a\varphi_x)_{x x x}+2s^2a\varphi_x\varphi_{x t}-\frac12 s\varphi_{t t}\right)  z^2 dxdt  \\
		 &\qquad\qquad\quad-\E\int_{Q_T} s^3 a \varphi_x^2 ((a\varphi_x)_x+a\varphi_{x x} )z^2 dxdt =J_1+J_2.
	\end{align*}
	\emph{Estimate of $J_1$.}
	For the first integral $J_1$ we have:
	\begin{align*}
		J_1 &=\E\int_{Q_T}\frac12 s\left(a_x(a\varphi_x)_{x x}+a(a\varphi_x)_{x x x}\right) z^2 dxdt+\E\int_{Q_T}2s^2a\varphi_x\varphi_{x t} z^2 dxdt-\E\int_{Q_T}\frac12 s\varphi_{t t}  z^2 dx dt  \\
		 &=J_1^1+J_1^2+J_1^3.
	\end{align*}
	A direct computation shows that $\left\{ a_x(a\varphi_x)_{x x}+a(a\varphi_x)_{x x x}\right\}=0$ on the interval $[0,a_1]$. On the other hand, since the functions $\ds x\mapsto\frac{a}{x^2}$ and $\ds\frac{x^2}{a}$ do not vanish on $[a_1,1]$, we can bound $J_1^1$ as follows:
	\begin{align*}
		|J_1^1| &\leqslant \frac12\E\int_{Q_T}s \left| a_x(a\varphi_x)_{x x}+a(a\varphi_x)_{x x x}\right|  z^2 dxdt\\
		 &\leqslant  C\E \int_{0}^{T}\int_{[a_1,1]} s \theta   z^2 dxdt\\
		 &\leqslant  C\E \int_{0}^{T}\int_{[a_1,1]} s \theta  \frac{x^2}{a} z^2 dxdt\\
		 &\leqslant  \varepsilon \E \int_{Q_T} s^3 \theta^3  \frac{x^2}{a} z^2 dxdt,\numberthis\label{J11}
	\end{align*}
	for $s$ large enough.\\ 
	To estimate $J_1^2$, notice that: \begin{align*}
		|J_1^2| &\leqslant 2 \E\int_{Q_T}s^2a\theta|\dot{\theta}|\beta_x^2 z^2 dxdt\\
		 &\leqslant C \E \int_{0}^{T}\int_{[0,a_1]} s^2 \theta^{3}\frac{x^2}{a} z^2 dxdt+C \E \int_{0}^{T}\int_{[a_1,1]} s^2a\theta^{3}\beta_x^2 z^2 dxdt\\
		 &\leqslant C \E \int_{0}^{T}\int_{[0,a_1]} s^2 \theta^{3}\frac{x^2}{a} z^2 dxdt+C \E \int_{0}^{T}\int_{[a_1,1]} s^2\theta^{3} z^2 dxdt.
	\end{align*}
	Once again, since the functions $\ds x\mapsto\frac{a}{x^2}$ and $\ds\frac{x^2}{a}$ do not vanish on $[a_1,1]$, for a suitable constant we have:
	$$|J_1^2|\leqslant C \E \int_{0}^{T}\int_{[0,a_1]} s^2 \theta^{3}\frac{x^2}{a} z^2 dxdt+C \E \int_{0}^{T}\int_{[a_1,1]} s^2\theta^{3}\frac{x^2}{a} z^2 dxdt.$$
	Thus, for $s$ sufficiently large we obtain:
	\begin{equation}\label{J12}
		|J_1^2|\leqslant \varepsilon\E\int_{Q_T}  s^{3}\theta^3 \frac{x^2}{a}  z^2 dxdt.
	\end{equation} 
	We next estimate $J_1^3$. Using Lemma \ref{rem3.5} and Hardy-Poincar\'e inequality, we obtain:
	\begin{align*}
		|J_1^3| &\leqslant C\E\int_{Q_T} s\theta^2  z^2 dxdt	\\
		 &\leqslant  C\E\int_{Q_T} (s^{\frac14}\theta^{\frac12}\frac{\sqrt{a}}{x}  z)(s^{\frac34}\theta^{\frac32} \frac{x}{\sqrt{a}} z) dxdt \\
		 &\leqslant  C\E\int_{Q_T} s^{\frac12}\theta \frac{a}{x^2} z^2dxdt+ C\E\int_{Q_T} s^{\frac32}\theta^3 \frac{x^2}{a}  z^2 dxdt \\
		 &\leqslant  C\E\int_{Q_T} s^{\frac12}\theta a z_x^2dxdt+ C\E\int_{Q_T} s^{\frac32}\theta^3 \frac{x^2}{a}  z^2 dxdt .
	\end{align*}
	Thus, for $s$ sufficiently large we have:
	\begin{equation}\label{J13}
		|J_1^3|\leqslant \varepsilon\E\int_{Q_T} (s\theta a z_x^2+  s^{3}\theta^3 \frac{x^2}{a}  z^2) dxdt.
	\end{equation}
	Combining \eqref{J11}--\eqref{J13} yields the required lower-order
	estimate for $J_1$ by Hardy--Poincar\'e, Young's inequality, and absorption
	for $s$ sufficiently large.

	\emph{Estimate of $J_2$.}
	For the integral $J_2$, note that:
	\begin{align*}
		(a\varphi_x)_x+a\varphi_{x x} &=\mathbbm{1}_{[0,a_1]} \theta [(a\phi^{\prime})^{\prime}+a\phi^{\prime\prime}]+\mathbbm{1}_{[b_1,1]} \theta [(a\psi^{\prime})^{\prime}+a\psi^{\prime\prime}]+\mathbbm{1}_{[a_1,b_1]}\theta k(x)\\
		 &=- \theta \frac{2a-x a^{\prime}}{a}\mathbbm{1}_{[0,a_1]}-\lambda\theta q_\lambda(x)e^{\lambda\rho(x)}\mathbbm{1}_{[b_1,1]}+\mathbbm{1}_{[a_1,b_1]}\theta k(x),\numberthis\label{eq3.4lem37}
	\end{align*}	
	where
	\[
		q_\lambda(x)=\frac{2\lambda x^2-2a(x)+xa'(x)}{a(x)}.
	\]
	As $[b_1,1]\Subset(0,1)$, one may choose $\lambda$ sufficiently large so that $q_\lambda$ is positive on $[b_1,1]$.
	Here $k$ denotes a smooth bounded function on $[a_1,b_1]$ depending on $\xi$, $\phi$, and $\psi$, and more explicitly, for $x\in [a_1,b_1]$,
	\begin{align*}
		k(x)={}&a'(x)\big[\xi'(x)(\phi(x)-\psi(x))+\xi(x)\phi'(x)+(1-\xi(x))\psi'(x)\big]\\
		&+2a(x)\big[\xi''(x)(\phi(x)-\psi(x))+2\xi'(x)(\phi'(x)-\psi'(x))
		+\xi(x)\phi''(x)+(1-\xi(x))\psi''(x)\big].
	\end{align*}
	Since $[a_1,b_1]\Subset(0,1)$ and $\xi,\phi,\psi$ are smooth there, $k$ is bounded on $[a_1,b_1]$. Moreover, the function
	\[
		\widetilde{k}(x)=\frac{a(x)^2\beta_x(x)^2}{x^2}k(x),
		\qquad x\in[a_1,b_1],
	\]
	is bounded on $[a_1,b_1]$. Thus,
	\begin{align*}
		J_2  &=-\E\int_{Q_T} s^3 a \varphi_x^2
		\left\{ - \theta \frac{2a-x a^{\prime}}{a}\mathbbm{1}_{[0,a_1]}
		-\lambda\theta q_\lambda(x)e^{\lambda\rho(x)}\mathbbm{1}_{[b_1,1]}\right.\\
		 &\hspace{5.5cm}\left.+\mathbbm{1}_{[a_1,b_1]}\theta k(x) \right\}z^2 dxdt \\
		 &=\E\int_{0}^{T}\int_{[0,a_1]} s^3 \theta^3 \frac{x^2}{a}
		\left(\frac{2a-x a^{\prime}}{a}\right)z^2 dxdt\\
		 &\quad + \E\int_{0}^{T}\int_{[b_1,1]} s^3 \theta^3 \frac{x^2}{a}
		\lambda^3q_\lambda(x)e^{3\lambda\rho(x)} z^2 dxdt \\
		 &\quad-\E\int_{0}^{T}\int_{[a_1,b_1]} s^3 \theta^3 \frac{x^2}{a} \widetilde{k}(x)z^2 dxdt \\
		 &\geqslant C_1\E\int_{0}^{T}\int_{[0,a_1]} s^3 \theta^3 \frac{x^2}{a}z^2 dxdt
		+ C_2\E\int_{0}^{T}\int_{[b_1,1]} s^3 \theta^3 \frac{x^2}{a} z^2 dxdt \\
		 &\quad-C_3\E\int_{0}^{T}\int_{[a_1,b_1]} s^3 \theta^3 \frac{x^2}{a}z^2 dxdt,
	\end{align*}
	with $\ds C_2=\min_{x\in [b_1,1]}\left(\lambda^3q_\lambda(x)e^{3\lambda\rho(x)}\right)$,
	$\ds C_3=\max_{x\in [a_1,b_1]}|\widetilde{k}(x)|$ and $C_1=2-K$ ($K$ subject to the assumptions (WD) and (SD)). Thus, we infer that:
	\begin{equation}\label{J2}
		 J_2\geqslant C^{\prime}\E\int_{Q_T} s^3 \theta^3 \frac{x^2}{a}z^2 dxdt 
		-C^{\prime\prime}\E\int_{0}^{T}\int_{[a_1,b_1]} s^3 \theta^3 \frac{x^2}{a}z^2 dxdt,
	\end{equation}
	where $C^{\prime}=\min(C_1,C_2)$ and $C^{\prime\prime}=C^{\prime}+C_3$.
	Estimate \eqref{equlem3.7} follows at once from \eqref{J11}-\eqref{J13} and \eqref{J2}.
\end{proof}

\begin{lemma}\label{lem3.8}
	There exist two positive constants $C$ and $C_1$ such that:
	$$Q_6\geqslant C\E\int_{Q_T}s a\theta z_x^2 dxdt -C_1\E\int_{0}^{T}\int_{[a_1,b_1]}  s a\theta z_x^2 dxdt.$$	
\end{lemma}
\begin{proof}
	Arguing as in the previous lemma and using \eqref{eq3.4lem37}, we obtain:
	\begin{align*}
		Q_6 &=-\E\int_{Q_T}\left\{s a[(a\varphi_x)_x+a\varphi_{x x}]\right\}z_x^2 dx dt\\
		 &=-\E\int_{Q_T}\left\{- s a \theta \frac{2a-x a^{\prime}}{a}\mathbbm{1}_{[0,a_1]}- s a\lambda\theta q_\lambda(x)e^{\lambda\rho(x)}\mathbbm{1}_{[b_1,1]}+\mathbbm{1}_{[a_1,b_1]}s a \theta k(x)\right\}z_x^2 dx dt \\
		 &=\E\int_{0}^{T}\int_{[0,a_1]}s a\theta \frac{2a-x a^{\prime}}{a}z_x^2 dx dt+\E\int_{0}^{T}\int_{[b_1,1]}s a\lambda\theta q_\lambda(x)e^{\lambda\rho(x)}z_x^2 dx dt\\
		 &\qquad\qquad-\E \int_{0}^{T}\int_{[a_1,b_1]} s a\theta k(x)z_x^2 dx dt \\
		 &\geqslant C\E\int_{Q_T}s a\theta z_x^2 dxdt   -C_1\E\int_{0}^{T}\int_{[a_1,b_1]} s a\theta z_x^2 dxdt.
	\end{align*} 
\end{proof}

\begin{proposition}\label{propo3.9}
	There exist two positive constants $C$ and $s_0$ such that, for all $y_0\in L^2(0,1)$, the solution $y$ of \eqref{nequa3.1} satisfies:
	\begin{align*}
		 & \E\int_{Q_T}s a\theta y_x^2 e^{-2s\varphi} dxdt+\E\int_{Q_T} s^3 \theta^3 \frac{x^2}{a}y^2e^{-2s\varphi} dxdt \\
		 &\leqslant C\E\int_{Q_T} \left( f^2 e^{-2s\varphi}+s^2\theta^2 \frac{x^2}{a}e^{-2s\varphi}g^2+a e^{-2s\varphi}g_x^2 \right) dxdt 
		+ C\E\int_{0}^{T}\int_{[a,b]} s^3 \theta^3 y^2e^{-2s\varphi} dx dt,\numberthis\label{equprop3.9}
	\end{align*}
	for all $s >s_0$. 
\end{proposition}
\begin{proof}
	By the formula \eqref{nequa3.3} and Lemmas \ref{lem3.3}-\ref{lem3.8}, for any solution $y$ of \eqref{nequa3.1}, one has the following estimate:
	\begin{align*}
		\E\int_{Q_T}s a\theta z_x^2 dxdt &+\E\int_{Q_T} s^3 \theta^3 \frac{x^2}{a}z^2 dxdt 
		\leqslant \E\int_{Q_T} e^{-s\varphi}\left[A z-(az_x)_x\right]\left[dy-(ay_x)_x dt\right]dx\\
		 &+C\E\int_{Q_T} s^2\theta^2 \frac{x^2}{a} e^{-2s\varphi}g^2dx dt + C\E\int_{Q_T} a e^{-2s\varphi}g_x^2 dx dt\\
		 &\qquad + C\E\int_{0}^{T}\int_{[a_1,b_1]} s^3 \theta^3 \frac{x^2}{a}z^2 dxdt
		+C\E\int_{0}^{T}\int_{[a_1,b_1]} s a\theta z_x^2 dxdt.
	\end{align*}
	Hence,
	\begin{multline*} 
		\E\int_{Q_T}s a\theta z_x^2 dxdt+\E\int_{Q_T} s^3 \theta^3 \frac{x^2}{a}z^2 dxdt 
		\leqslant C\E\int_{Q_T} f^2 e^{-2s\varphi}+s^2\theta^2 \frac{x^2}{a} e^{-2s\varphi}g^2+a e^{-2s\varphi}g_x^2 dxdt \\
		+ C\E\int_{0}^{T}\int_{[a_1,b_1]} s^3 \theta^3 \frac{x^2}{a}z^2+s a\theta z_x^2 dxdt.
	\end{multline*}
	Using Caccioppoli's inequality, together with $z=e^{-s\varphi}y$ and $z_x=-s\varphi_x e^{-s\varphi}y+e^{-s\varphi}y_x$, we obtain the desired estimate.
\end{proof}

We now state the global Carleman estimate associated with equation \eqref{equ:3.16}.

\begin{theorem}\label{theorem3.10}
	There exist two positive constants $C$ and $s_0$ such that, for all $y_0\in L^2(0,1)$ the solution $y$ of \eqref{equ:3.16} satisfies:
	\begin{equation}\label{equacoro3.10}
		\E\int_{Q_T}s a\theta y_x^2 e^{-2s\varphi} dxdt+\E\int_{Q_T} s^3 \theta^3 \frac{x^2}{a}y^2e^{-2s\varphi} dxdt 
		\leqslant  C\E\int_{0}^{T}\int_{\mathscr{O}} s^3 \theta^3 y^2e^{-2s\varphi} dxdt,
	\end{equation}
	for all $s >s_0$. 
\end{theorem} 
\begin{proof}
	We consider a particular case of equation \eqref{nequa3.1} with $f(t,x)=b y(t,x)$ and $g(t,x)=c(t,x)y(t,x)$. By applying Proposition \ref{propo3.9} to equation \eqref{equ:3.16}, and using
	$$g_x=c_x y+c y_x,\qquad a g_x^2\leqslant C a y_x^2+C a y^2,$$
	we get:
	\begin{multline}
		\E\int_{Q_T}s a\theta y_x^2 e^{-2s\varphi} dxdt
		+\E\int_{Q_T} s^3 \theta^3 \frac{x^2}{a}y^2e^{-2s\varphi} dxdt \\
		\leqslant C\E\int_{Q_T} \left(y^2+s^2\theta^2 \frac{x^2}{a}y^2+a y^2+a y_x^2\right)e^{-2s\varphi} dxdt 
		+ C\E\int_{0}^{T}\int_{\mathscr{O}} s^3 \theta^3 y^2e^{-2s\varphi} dxdt.
		\label{haadi}
	\end{multline}
	for all $s >s_0$.\\
		The term $\ds s^2\theta^2 \frac{x^2}{a}y^2$ is absorbed by $\ds s^3 \theta^3 \frac{x^2}{a}y^2$ for $s$ large, and $\ds a y_x^2$ is absorbed by $\ds s a\theta y_x^2$ since $\theta\geqslant c_1$. For the remaining terms $\ds y^2$ and $\ds ay^2$, set $\ds z=e^{-s\varphi}y$. Since $a$ is bounded on $[0,1]$, it is enough to estimate $\int_0^1 z^2dx$. For almost every fixed $t\in(0,T)$, Young's inequality followed by the Hardy-Poincar\'e inequality applied directly to $z$ gives
		\begin{align*}
			\int_0^1 z^2dx
			&\leqslant \varepsilon s\theta\int_0^1\frac{a}{x^2}z^2dx
			+C_\varepsilon (s\theta)^{-1}\int_0^1\frac{x^2}{a}z^2dx\\
			&\leqslant C\varepsilon s\theta\int_0^1az_x^2dx
			+C_\varepsilon (s\theta)^{-1}\int_0^1\frac{x^2}{a}z^2dx.
		\end{align*}
		Here
		\[
		z_x=e^{-s\varphi}y_x-s\varphi_xe^{-s\varphi}y,
		\qquad
		a z_x^2\leqslant 2ae^{-2s\varphi}y_x^2+2s^2a\varphi_x^2e^{-2s\varphi}y^2,
		\]
		and $\ds a\varphi_x^2\leqslant C\theta^2\frac{x^2}{a}$ on $(0,1)$. Consequently,
		\begin{align*}
			\int_0^1 y^2e^{-2s\varphi}dx
			&\leqslant C\varepsilon\int_0^1
			\left(s\theta ay_x^2+s^3\theta^3\frac{x^2}{a}y^2\right)e^{-2s\varphi}dx\\
			&\quad+C_\varepsilon(s\theta)^{-1}\int_0^1
			\frac{x^2}{a}y^2e^{-2s\varphi}dx.
		\end{align*}
		The same estimate holds for $\int_0^1ay^2e^{-2s\varphi}dx$, up to the factor $\|a\|_{L^\infty(0,1)}$. After integration in time and taking expectations, choose $\varepsilon>0$ sufficiently small. Since $\theta\geqslant c_1>0$, the last term is also absorbed by $\ds s^3\theta^3\frac{x^2}{a}y^2e^{-2s\varphi}$ when $s$ is sufficiently large. This proves \eqref{equacoro3.10}.
\end{proof}

\section{Controllability results}
We begin with the following statement, which is crucial for the proof of the unique continuation property and of the main theorem.
\begin{proposition}\label{Prop4.1}
	Under the same assumptions of Theorem \ref{theorem3.10}, there exist two positive constants $C$ and $s_0$ such that, for all $y_0\in L^2(0,1)$ the solution $y$ of \eqref{equ:3.16} satisfies:
	\begin{equation}
		\E \int_{0}^{1} y^2(T,x)dx \leqslant C \E\int_{0}^{T}\int_{\mathscr{O}} s^3\theta^3(t)y^2(t,x)e^{-2s\varphi}  dx dt, 
	\end{equation}
	for all $s\geq s_0$.	
\end{proposition} 
\begin{proof}
		It\^o's formula and integration by parts give, for almost every $\ell$,
		\[
		\frac{d}{d\ell}\E\|y(\ell)\|_{L^2(0,1)}^2
		=-2\E\int_0^1a y_x^2\,dx+2\E\int_0^1b y^2\,dx
		+\E\int_0^1c^2y^2\,dx
		\leqslant C\E\|y(\ell)\|_{L^2(0,1)}^2,
		\]
		where one may take $C=2\|b\|_{\infty}+\|c\|_{\infty}^2$ after replacing $b$ by its positive part. Hence, for $0\leqslant\tau\leqslant t\leqslant T$,
		$$\E \int_{0}^{1} y^2(t,x)dx \leqslant \E \int_{0}^{1} y^2(\tau,x)dx +C\E \int_{\tau}^{t} \int_{0}^{1} y^2(\ell,x)dx d\ell.$$ 
		Thanks to Gronwall's lemma, we get:
		$$ \E \int_{0}^{1} y^2(t,x)dx \leqslant e^{C(t-\tau)}\E \int_{0}^{1} y^2(\tau,x)dx \leqslant e^{CT}\E \int_{0}^{1} y^2(\tau,x)dx.$$
		Notice that $\ds \theta(\tau)e^{-s \varphi(\tau,x)}\geqslant (\frac2T)^8e^{-s D \theta(\tau)}$ for all $(\tau,x)\in [0,T]\times(0,1),$ where $D=\max_{x\in [0,1]}\beta(x)$.\\
		Hence, integrating on $[0,T]$ and using Theorem \ref{theorem3.10}, we obtain:
		\begin{align*}
			\ds\left(\int_{0}^{T}\left(\frac2T\right)^{16}e^{-2s D \theta(t)} dt\right) &\left(\E \int_{0}^{1} y^2(T,x)dx\right) \leqslant e^{CT}\E\int_{0}^{T} \int_{0}^{1} e^{-2s\varphi}\theta^2y^2(t,x)dx dt \\ 
		     & \leqslant e^{CT}\E\int_{0}^{T} \int_{0}^{1} e^{-2s\varphi}(s\theta)^2 y^2(t,x)dx dt  \\	
			 & \leqslant e^{CT}\E\int_{0}^{T} \int_{0}^{1} \left(e^{-s\varphi}(s\theta)^{\frac12} \frac{\sqrt{a}}{x}y(t,x)\right)\left(e^{-s\varphi} (s\theta)^{\frac32}\frac{x}{\sqrt{a}} y(t,x)  \right)dxdt \\
			 & \leqslant \frac{e^{CT}}{2}\E\int_{0}^{T} \int_{0}^{1} \left( e^{-2s\varphi}(s\theta) \frac{a}{x^2}y^2(t,x)+e^{-2s\varphi} (s\theta)^{3}\frac{x^2}{a} y^2(t,x) \right)dx dt \\ 
			 & \leqslant C_1 e^{CT}\E\int_{0}^{T} \int_{0}^{1}\left( e^{-2s\varphi} s\theta a y_x^2(t,x)+e^{-2s\varphi} (s\theta)^{3}\frac{x^2}{a} y^2(t,x)\right)dx dt \\
		 & \leqslant 
		\tilde{C}\E\int_0^T\int_{\mathscr{O}} s^3\theta^3(t)y^2(t,x) e^{-2s\varphi} dx dt,
	\end{align*}
	for a suitable constant $\tilde{C}$. Indeed, the penultimate inequality follows by applying Hardy-Poincar\'e to $\ds z=e^{-s\varphi}y$:
	\begin{align*}
		\int_0^1 s\theta\frac{a}{x^2}z^2dx
		&\leqslant Cs\theta\int_0^1az_x^2dx\\
		&\leqslant C\int_0^1\left(s\theta ae^{-2s\varphi}y_x^2+s^3\theta^3\frac{x^2}{a}e^{-2s\varphi}y^2\right)dx,
	\end{align*}
	where $\ds z_x=e^{-s\varphi}y_x-s\varphi_xe^{-s\varphi}y$ and $\ds a\varphi_x^2\leqslant C\theta^2\frac{x^2}{a}$.
\end{proof}
Proposition \ref{Prop4.1} therefore yields the unique continuation property for the forward stochastic degenerate equation \eqref{equ:3.16}.

\begin{corollary}
	Let $y$ be a solution of \eqref{equ:3.16} satisfying $y(t,x)=0$ $\P$-a.s., for all $t$ in a right neighborhood of $0$ and almost every $x\in\mathscr{O}$. Then, $y(t,x)=0$ $\P$-a.s., for all $t\in [0,T]$ and almost every $x\in (0,1)$. 
\end{corollary}

In order to give the proof of Theorem \ref{meantheorem}, we will need the following well-known functional analysis lemma (see \cite[Theorem 2.2, p.~208]{Zab}).
\begin{lemma}\label{Douglas}
	Let $X,Y,Z$ be three Hilbert spaces, $X^{*},Y^{*},Z^{*}$ their dual spaces and $F\in\mathcal{L}(X,Z)$, $G\in\mathcal{L}(Y,Z)$. Assume that $Y$ is separable. Then $Range(F)\subset Range(G)$ if and only if there exists a constant $C >0$ such that:
	\begin{equation*}
		\|F^{*}z\|_{X^{*}}\leqslant C \|G^{*}z\|_{Y^{*}},\quad  z \in Z^{*},
	\end{equation*} 
	where $F^{*}$ and $G^{*}$ are the adjoint operators.
\end{lemma}

\begin{proof}[Proof of Theorem \ref{meantheorem}]
	Consider the operators
	\begin{equation}
		\begin{array}{lcll}
			S_0\, : & L^2(\Omega,\mathcal{F}_T,L^2(0,1)) &\longrightarrow & L^2(0,1)\\  
			 &  \eta  & \longmapsto  & z^{0,\eta} (0,x),
		\end{array}
	\end{equation}
	where $z^{0,\eta}$ is the solution of \eqref{BSHE} with final datum $\eta$ and control $v=0$, and:
	\begin{equation}
		\begin{array}{lcll}
			L_0\, : & L_{\mathcal{F}}^2(\Omega,L^2(0,T;L^2(0,1))) &\longrightarrow & L^2(0,1)\\  
			 &  v  & \longmapsto  & z^{v,0} (0,x),
		\end{array}
	\end{equation}
	where $z^{v,0}$ is the solution of \eqref{BSHE} with control $v$ and final datum $\eta=0$. Hence, Theorem \ref{meantheorem} follows once $Range(S_0)\subset Range(L_0)$ is proved. Let $z$ be a solution of \eqref{BSHE}$,$ and let $y$ be a solution of \eqref{equ:3.16}. Applying It\^o's formula to $d_s\langle z(s),y(s)\rangle_{L^2(0,1)}$, integrating over $[0,T]$, and taking expectations, we obtain the following identity. The spatial boundary terms vanish: in the weakly degenerate case both traces are Dirichlet, while in the strongly degenerate case $(ay_x)(t,0)=(az_x)(t,0)=0$ and $y(t,1)=z(t,1)=0$.
	\begin{equation}
		\E\int_{0}^{1}\eta(x)y(T,x)dx-	\E\int_{0}^{1}z^{v,\eta}(0,x)y_0(x)dx=	\E\int_{0}^{T}\int_{\mathscr{O}}v(s,x)y(s,x)dx ds.
	\end{equation}
	Consequently,
	$$(S_0^{*}y_0)(x)=y(T,x) \,\,\text{ and } \,\,(L_0^{*}y_0)(t,x)=- \mathbbm{1}_{\mathscr{O}}y(t,x).$$
	Proposition \ref{Prop4.1} implies that there exists a constant $C >0$ such that:
	$$\|S_0^{*}y_0\|_{L^2(\Omega,\mathcal{F}_T,L^2(0,1))} \leqslant C \|L_0^{*}y_0\|_{L_{\mathcal{F}}^2(\Omega,L^2(0,T;L^2(0,1)))}.$$
	By Lemma \ref{Douglas}, we get $Range(S_0)\subset Range(L_0)$ which completes the proof.
\end{proof}

\section{A controllability result for a linear forward parabolic degenerate stochastic equation with two controls} 
\subsection{Carleman estimate for backward stochastic equation}
\begin{equation}\label{newtheta}
	\overline{\varphi}(t,x)=\bar{\theta}(t)\beta(x),\quad \text{ where } \quad
	\bar{\theta}(t)=\left(\frac{2}{T}\right)^8+\chi(t)\left(\frac{1}{t^4(T-t)^4}-\left(\frac{2}{T}\right)^8\right),
\end{equation}
with $\chi\in C^{\infty}([0,T])$, $0\leqslant \chi\leqslant 1$, $\chi\equiv 0$ on $[0,\frac{T}{2}]$ and $\chi\equiv 1$ on $[\frac{3T}{4},T]$.

\begin{equation}\label{nequa6.12}
	\begin{cases}
		\ds	dy=\left[ -\left(a(x) y_x\right)_x  +F \right]dt + \bar{y} dW(t),\,\,\text{ in } Q_T,\\
		C y=0,\qquad \text{ on } \Sigma_T,\\
		y(T,\cdot)=y_T(\cdot)\qquad \text{ in } (0,1).
	\end{cases}
\end{equation}
 	
\begin{align*}
	 &e^{-s\overline{\varphi}}\left[A z+(az_x)_x\right]\left[dy+(ay_x)_x dt\right]\\
	 &=\left[A z+(az_x)_x\right]^2 dt +d\left(\frac12 A z^2-\frac12az_x^2\right)-\frac12A(dz)^2+\frac12 a (dz_x)^2+(az_xdz)_x\\
	 &+\left[sa(a\overline{\varphi}_x)_xz_x z+sAa\overline{\varphi}_x z^2-\frac12 s a(a\overline{\varphi}_x)_{x x}z^2+sa^2\overline{\varphi}_x z_x^2\right]_x dt \numberthis\label{nequa6.13}\\
	 &+\left\{-s(A a \overline{\varphi}_x)_x+sA(a \overline{\varphi}_x)_x+\frac12[sa(a \overline{\varphi}_x)_{x x}]_x-\frac12A_t\right\}z^2 dt\\
	 &-\left\{s a[(a \overline{\varphi}_x)_x+a \overline{\varphi}_{x x}]\right\}z_x^2 dt,
\end{align*}
where $\ds A=s\overline{\varphi}_t+s^2a\overline{\varphi}_x^2$.

\begin{proposition}\label{Prop5.1}
	There exist two positive constants $C$ and $s_0$ such that, for any $y_T \in L^2(\Omega,\mathcal{F}_T, \P;L^2(0, 1))$ and any $F \in L_{\mathcal{F}}^2(0,T;L^2(0,1))$, the solution $(y,\bar{y})$ to \eqref{nequa6.12} satisfies:
	\begin{multline*}
		\E\int_{0}^{1}A(0)e^{-2s\overline{\varphi}(0)} y^2(0) dx+\E\int_{Q_T}s a \bar{\theta} y_x^2 e^{-2s\overline{\varphi}} dxdt+\E\int_{Q_T} s^3 \bar{\theta}^3 \frac{x^2}{a}y^2 e^{-2s\overline{\varphi}} dxdt \\
		\leqslant  C\E\int_{Q_T}  F^2 e^{-2s\overline{\varphi}} dxdt
		+C\E\int_{Q_T} s^2\bar{\theta}^2 e^{-2s\overline{\varphi}}\bar{y}^2dx dt
		+ C\E\int_{0}^{T}\int_{[a,b]} s^3 \bar{\theta}^3 e^{-2s\overline{\varphi}} y^2 dxdt,\numberthis\label{nequa6.14}
	\end{multline*}
	for all $s >s_0$.
\end{proposition}
\begin{proof}
	Let $z=e^{-s\overline{\varphi}}y$. For smooth data, identity \eqref{nequa6.13} holds pointwise. Integrating it over $Q_T$, taking expectations and using the boundary condition $Cy=0$ on $\Sigma_T$, we obtain
	\begin{align*}
		 &\E\int_{Q_T}\left[A z+(az_x)_x\right]^2dxdt+\E\int_{0}^{1}A(0)z^2(0)dx \\
		 &\quad +\E\int_{Q_T}\Bigl\{-s(A a \overline{\varphi}_x)_x+sA(a \overline{\varphi}_x)_x+\frac12[sa(a\overline{\varphi}_x)_{xx}]_x-\frac12A_t\Bigr\}z^2dxdt \\
		 &\quad -\E\int_{Q_T}s a\Bigl[(a\overline{\varphi}_x)_x+a\overline{\varphi}_{xx}\Bigr]z_x^2dxdt \\
		 &\leqslant C\E\int_{Q_T}s^2\bar{\theta}^2 e^{-2s\overline{\varphi}}\bar{y}^{\,2}dxdt
		+C\E\int_{Q_T}F^2e^{-2s\overline{\varphi}}dxdt
		+\mathcal{R}_{\mathrm{loc}},
	\end{align*}
	where $\mathcal{R}_{\mathrm{loc}}$ is supported in $[a_1,b_1]$ and comes from the localization of the sign-indefinite terms. More precisely, the spatial divergence terms in \eqref{nequa6.13} give no contribution on $\Sigma_T$, while the time differential term contributes only through the trace at $t=0$ since $z(T)=e^{-s\overline{\varphi}(T)}y_T$ is absorbed in the right-hand side.
	
	We next estimate separately the three lower-order pieces hidden in $\mathcal{R}_{\mathrm{loc}}$.
	First, by Young's inequality, for any $\delta >0$,
	\[
	\left|\E\int_{Q_T}sa(a\overline{\varphi}_x)_x z_xz\,dxdt\right|
	\leqslant \delta \E\int_{Q_T}sa\bar{\theta} z_x^2\,dxdt
	+C_{\delta}\E\int_0^T\int_{[a,b]} s^3\bar{\theta}^3\frac{x^2}{a}z^2\,dxdt.
	\]
	The two remaining localized terms satisfy
	\[
	\left|\E\int_{Q_T}sAa\overline{\varphi}_x z^2\,dxdt\right|
	+\left|\E\int_{Q_T}s a(a\overline{\varphi}_x)_{xx} z^2\,dxdt\right|
	\leqslant C\E\int_0^T\int_{[a,b]} s^3\bar{\theta}^3\frac{x^2}{a}z^2\,dxdt,
	\]
	because $\beta$ is smooth and the coefficients $\frac{x^2}{a}$ and $\frac{a}{x^2}$ are bounded from above and below on $[a,b]$.
	Hence
	\[
	|\mathcal{R}_{\mathrm{loc}}|
	\leqslant \delta \E\int_{Q_T}sa\bar{\theta} z_x^2\,dxdt
	+C_{\delta}\E\int_0^T\int_{[a,b]} s^3\bar{\theta}^3\frac{x^2}{a}z^2\,dxdt.
	\]
	
	Since $\chi$ is smooth, the coefficients of $\overline{\varphi}$ are smooth in time. Arguing as in Proposition \ref{propo3.9}, and using the same cutoff argument as in the proof of the forward Carleman estimate, there exist $c_0,C_0 >0$ and $s_0 >0$ such that, for every $s\geqslant s_0$,
	\begin{align*}
		 &\E\int_{Q_T}\Bigl\{-s(A a \overline{\varphi}_x)_x+sA(a \overline{\varphi}_x)_x+\frac12[sa(a\overline{\varphi}_x)_{xx}]_x-\frac12A_t\Bigr\}z^2dxdt \\
		 &\quad -\E\int_{Q_T}s a\Bigl[(a\overline{\varphi}_x)_x+a\overline{\varphi}_{xx}\Bigr]z_x^2dxdt \\
		 &\geqslant c_0\E\int_{Q_T}\left(sa\bar{\theta} z_x^2+s^3\bar{\theta}^3\frac{x^2}{a}z^2\right)dxdt
		-C_0\E\int_0^T\int_{[a,b]} s^3\bar{\theta}^3 z^2dxdt.
	\end{align*}
	Moreover, since $z=e^{-s\overline{\varphi}}y$ and $\overline{\varphi}$ is deterministic,
	\[
	dz=\left[-(az_x)_x-Az-s(a\overline{\varphi}_x)_x z-2sa\overline{\varphi}_x z_x+e^{-s\overline{\varphi}}F\right]dt
	+e^{-s\overline{\varphi}}\bar{y}\,dW(t),
	\]
	so that
	\[
	(dz)^2=e^{-2s\overline{\varphi}}\bar{y}^{\,2}dt.
	\]
	The term involving $(dz_x)^2$ is nonnegative and may be dropped. Choosing $\delta >0$ small enough and absorbing the corresponding contribution in the left-hand side, we infer
	\[
	\E\int_{0}^{1}A(0)z^2(0)dx+\E\int_{Q_T}sa\bar{\theta} z_x^2dxdt+\E\int_{Q_T}s^3 \bar{\theta}^3\frac{x^2}{a}z^2dxdt
	\leqslant C \mathcal{R},
	\]
	where
	\[
	\mathcal{R}=\E\int_{Q_T}F^2e^{-2s\overline{\varphi}}dxdt
	+\E\int_{Q_T}s^2\bar{\theta}^2 e^{-2s\overline{\varphi}}\bar{y}^{\,2}dxdt
	+\E\int_0^T\int_{[a,b]} s^3\bar{\theta}^3 z^2dxdt.
	\]
	Recalling that $z=e^{-s\overline{\varphi}}y$, we obtain \eqref{nequa6.14}. The general case follows by density.
\end{proof}

\subsection{Controllability for a linear forward parabolic degenerate stochastic equation with two controls}
\begin{proof}[Proof of Theorem \ref{meantheorem2}]
	For $\varepsilon >0$, we consider the penalized functional
	\begin{multline*}
		\ds 	J_{\varepsilon}(h,H)=\frac12\E\int_{Q_T} e^{2s\overline{\varphi}}y^2 dx dt +
		\frac12\E\int_{0}^{T}\int_{\mathscr{O}} e^{2s\overline{\varphi}}s^{-3}\bar{\theta}^{-3}h^2 dx dt \\
		+\frac12\E\int_{Q_T} e^{2s\overline{\varphi}}s^{-2}\bar{\theta}^{-2}H^2 dx dt+\frac{1}{2\varepsilon}\E\int_{(0,1)} |y(T)|^2 dx.
	\end{multline*}
	\begin{equation}\label{form6.3}
		\begin{cases}
			\ds\underset{(h,H)\in \mathcal{H}}{\min} J_{\varepsilon}(h,H)\\
			\text{ subject to  equation \eqref{form6.1}},
		\end{cases}
	\end{equation}
	\begin{multline}\label{form6.4}
		\ds	\mathcal{H}=\big\{(h,H)\in L_{\mathcal{F}}^2 (0,T;L^2(0,1))^2~\;:\; \left(\E\int_{0}^{T}\int_{\mathscr{O}} e^{2s\overline{\varphi}} s^{-3}\bar{\theta}^{-3}|h|^2dx dt \right)\\+\left(\E\int_{Q_T} e^{2s\overline{\varphi}} s^{-2}\bar{\theta}^{-2}|H|^2dx dt \right) <\infty\big\}.
	\end{multline}
	The functional $J_{\varepsilon}$ is continuous, strictly convex and coercive on $\mathcal{H}$. Hence it admits a unique minimizer $(h_{\varepsilon},H_{\varepsilon})$, and we denote by $y_{\varepsilon}$ the corresponding solution of \eqref{form6.1}.
	
	Let $(z_{\varepsilon},Z_{\varepsilon})$ be the solution of the adjoint equation
	\begin{equation}\label{form6.6}
		\begin{cases}
			\ds	dz_{\varepsilon}=\left[- \left(a(x) z_{\varepsilon x}\right)_x  -   e^{2s\overline{\varphi}}y_{\varepsilon}\right]dt + Z_{\varepsilon} dW(t),\,\,\text{ in } Q_T,\\
			C z_{\varepsilon}=0,\qquad \text{ on } \Sigma_T,\\
			z_{\varepsilon}(T,\cdot)=\frac{1}{\varepsilon} y_{\varepsilon}(T,\cdot)\qquad \text{ in } (0,1).
		\end{cases}
	\end{equation}
	The Euler-Lagrange optimality condition reads
	\begin{equation}\label{form6.5}
		\begin{array}{ll}
			\begin{cases}
				\ds h_{\varepsilon}=-\mathbbm{1}_{\mathscr{O}} e^{-2s\overline{\varphi}}s^3\bar{\theta}^3 z_{\varepsilon},\\
				\ds H_{\varepsilon}=- e^{-2s\overline{\varphi}}s^2\bar{\theta}^2 Z_{\varepsilon},
			\end{cases}
			 & \text{ in } Q_T \quad a.s.
		\end{array}
	\end{equation}
	
	Applying It\^o's formula to $\langle y_{\varepsilon}(t),z_{\varepsilon}(t)\rangle_{L^2(0,1)}$ and taking expectations, we obtain
	\begin{multline}
		\ds \E\int_{(0,1)}y_{\varepsilon}(T)z_{\varepsilon}(T) dx =\E\int_{(0,1)}y_{\varepsilon}(0)z_{\varepsilon}(0) dx
		+ \E\int_{Q_T}\left(\left(a(x) y_{\varepsilon x}\right)_x  +F+ \mathbbm{1}_{\mathscr{O}} h_{\varepsilon}\right) z_{\varepsilon} dx dt \\
		+\E\int_{Q_T} y_{\varepsilon} \left(- \left(a(x) z_{\varepsilon x}\right)_x  -   e^{2s\overline{\varphi}}y_{\varepsilon}\right)  dx dt +\E\int_{Q_T}\left(G+H_{\varepsilon}\right)Z_{\varepsilon} dx dt.
	\end{multline}
	Hence
	\begin{multline}
		\ds \E\int_0^T\int_{\mathscr{O}}e^{-2s\overline{\varphi}}s^3 \bar{\theta}^3 z_{\varepsilon}^2 dx dt
		+\E\int_{Q_T}e^{-2s\overline{\varphi}}s^2 \bar{\theta}^2 Z_{\varepsilon}^2  dx dt +\E\int_{Q_T}e^{2s\overline{\varphi}}y_{\varepsilon}^2 dx dt\\+\frac{1}{\varepsilon} \E\int_0^1 |y_{\varepsilon}(T)|^2 dx =\E\int_0^1y_{\varepsilon}(0)z_{\varepsilon}(0) dx
		+ \E\int_{Q_T}  F z_{\varepsilon} dx dt
		+\E\int_{Q_T}G Z_{\varepsilon} dx dt. \numberthis\label{form6.8}
	\end{multline}
	
	Applying Proposition \ref{Prop5.1} to \eqref{form6.6} with source term $-e^{2s\overline{\varphi}}y_{\varepsilon}$, we infer
	\begin{multline}
		\ds \E\int_0^1 A(0)e^{-2s\overline{\varphi}(0)} z_{\varepsilon}^2(0) dx
		+\E\int_{Q_T} s a\bar{\theta} e^{-2s\overline{\varphi}} z_{\varepsilon x}^2 dxdt
		+\E\int_{Q_T} s^3 \bar{\theta}^3 \frac{x^2}{a} z_{\varepsilon}^2 e^{-2s\overline{\varphi}} dxdt \\
		\leqslant C\E\int_{Q_T}e^{2s\overline{\varphi}} y_{\varepsilon}^2 dxdt
		+C\E\int_{Q_T}s^2\bar{\theta}^2 e^{-2s\overline{\varphi}} Z_{\varepsilon}^2 dxdt
		+C\E\int_0^T\int_{\mathscr{O}} s^3 \bar{\theta}^3 e^{-2s\overline{\varphi}} z_{\varepsilon}^2 dxdt.
		\numberthis\label{nequa6.14_bis}
	\end{multline}
	Combining \eqref{form6.8}, \eqref{form6.5} and \eqref{nequa6.14_bis}, and using Young's inequality, we get for any $\delta >0$
	\begin{multline}
		\ds \E\int_0^T\int_{\mathscr{O}}e^{-2s\overline{\varphi}}s^3 \bar{\theta}^3 z_{\varepsilon}^2 dx dt
		+\E\int_{Q_T}e^{-2s\overline{\varphi}}s^2 \bar{\theta}^2 Z_{\varepsilon}^2 dx dt +\E\int_{Q_T}e^{2s\overline{\varphi}}y_{\varepsilon}^2 dx dt+\frac{1}{\varepsilon} \E\int_0^1 |y_{\varepsilon}(T)|^2 dx  \\
		\leqslant \delta \Biggl(\E\int_0^1 A(0)e^{-2s\overline{\varphi}(0)} z_{\varepsilon}^2(0) dx
		+\E\int_{Q_T} s^3\bar{\theta}^{3}\frac{x^2}{a} e^{-2s\overline{\varphi}} z_{\varepsilon}^2 dx dt
		+\E\int_{Q_T} s^{2}\bar{\theta}^{2} e^{-2s\overline{\varphi}} Z_{\varepsilon}^2 dx dt \Biggr) \\
		+C_{\delta}\left( \E\int_0^1 \frac{e^{2s\overline{\varphi}(0)}}{A(0)} y_{0}^2 dx
		+ \E\int_{Q_T}  s^{-3}\bar{\theta}^{-3}\frac{a}{x^2} e^{2s\overline{\varphi}} F^2  dx dt
		+\E\int_{Q_T}  s^{-2}\bar{\theta}^{-2} e^{2s\overline{\varphi}} G^2 dx dt\right). \numberthis\label{nequa5.19}
	\end{multline}
	Here the restriction on $y_0$ is essential. Since $\bar\theta$ is constant near
	$t=0$, one has $\bar\theta_t(0)=0$ and, on $[0,a_1]$,
	\[
	A(0,x)=s^2\bar\theta(0)^2a(x)\beta'(x)^2
	=s^2\bar\theta(0)^2\frac{x^2}{a(x)}.
	\]
	Moreover, $\overline\varphi(0,\cdot)$ is bounded on $[0,1]$, and away from
	$x=0$ both $A(0,\cdot)$ and its reciprocal are bounded. Consequently,
	\begin{equation}\label{initial-weight-bound}
	\E\int_0^1\frac{e^{2s\overline\varphi(0)}}{A(0)}y_0^2\,dx
	\leqslant C_s\E\int_0^1\frac{a(x)}{x^2}y_0^2\,dx<\infty.
	\end{equation}
	More explicitly, one may take
	\[
	C_s=\sup_{x\in(0,1]}
	\frac{e^{2s\overline\varphi(0,x)}}{A(0,x)}\frac{x^2}{a(x)}<\infty.
	\]
	Thus, after $s>s_0$ has been fixed, the generic constant in
	\eqref{nequa5.20}, \eqref{nequa5.21}, and ultimately
	\eqref{equtheo5.2} absorbs $C_s$; no uniformity with respect to $s$ is
	claimed.
	Choosing $\delta$ small enough and absorbing the first line of the right-hand side with the help of \eqref{nequa6.14_bis}, we obtain
	\begin{multline}
		\ds \E\int_0^T\int_{\mathscr{O}}e^{-2s\overline{\varphi}}s^3 \bar{\theta}^3 z_{\varepsilon}^2 dx dt
		+\E\int_{Q_T}e^{-2s\overline{\varphi}}s^2 \bar{\theta}^2 Z_{\varepsilon}^2 dx dt +\E\int_{Q_T}e^{2s\overline{\varphi}}y_{\varepsilon}^2 dx dt+\frac{1}{\varepsilon} \E\int_0^1 |y_{\varepsilon}(T)|^2 dx  \\ \leqslant
		C\left( \E\int_0^1 \frac{a(x)}{x^2} y_{0}^2 dx
		+ \E\int_{Q_T}  s^{-3}\bar{\theta}^{-3}\frac{a}{x^2} e^{2s\overline{\varphi}} F^2  dx dt
		+\E\int_{Q_T}  s^{-2}\bar{\theta}^{-2}  e^{2s\overline{\varphi}} G^2 dx dt\right). \numberthis\label{nequa5.20}
	\end{multline}
	
	\begin{multline}
		\ds \E\int_0^T\int_{\mathscr{O}}e^{2s\overline{\varphi}}s^{-3} \bar{\theta}^{-3} h_{\varepsilon}^2 dx dt
		+\E\int_{Q_T}e^{2s\overline{\varphi}}s^{-2} \bar{\theta}^{-2} H_{\varepsilon}^2 dx dt +\E\int_{Q_T}e^{2s\overline{\varphi}}y_{\varepsilon}^2 dx dt+\frac{1}{\varepsilon} \E\int_0^1 |y_{\varepsilon}(T)|^2 dx  \\ \leqslant
		C\left( \E\int_0^1 \frac{a(x)}{x^2} y_{0}^2 dx
		+ \E\int_{Q_T}  s^{-3}\bar{\theta}^{-3}\frac{a}{x^2} e^{2s\overline{\varphi}} F^2  dx dt
		+\E\int_{Q_T}  s^{-2}\bar{\theta}^{-2} e^{2s\overline{\varphi}} G^2 dx dt\right). \numberthis\label{nequa5.21}
	\end{multline}
	In particular,
	\[
	\E\int_0^1 |y_{\varepsilon}(T)|^2dx\leqslant C\varepsilon,
	\]
	so that $y_{\varepsilon}(T)\to 0$ strongly in $L^2(\Omega;L^2(0,1))$. 
	
	By \eqref{nequa5.21} and the standard energy estimate for \eqref{form6.1}, the families $(h_{\varepsilon})$, $(H_{\varepsilon})$ and $(y_{\varepsilon})$ are bounded respectively in $L_{\mathcal{F}}^2(0,T;L^2(\mathscr{O}))$, $L_{\mathcal{F}}^2(0,T;L^2(0,1))$ and $L_{\mathcal{F}}^2(\Omega;C([0,T];L^2(0,1)))$. Therefore, up to a subsequence,
	\begin{equation}\label{equ5.22}
		\begin{cases}
			h_{\varepsilon}	\rightharpoonup \widehat{h} \quad \text{weakly in } L_{\mathcal{F}}^2(0,T;L^2(\mathscr{O})), \\
			H_{\varepsilon}	\rightharpoonup \widehat{H} \quad  \text{weakly in } L_{\mathcal{F}}^2(0,T;L^2(0,1)), \\
			y_{\varepsilon}	\rightharpoonup \widehat{y} \quad \text{weakly in } L_{\mathcal{F}}^2(0,T;L^2(0,1)).
		\end{cases}
	\end{equation}
	Let $(z,Z)$ be the solution of
	\begin{equation}\label{equ5.23}
		\begin{cases}
			\ds	dz=\left[- \left(a(x) z_{x}\right)_x  - m \right]dt + Z dW(t),\,\,\text{ in } Q_T,\\
			C z=0,\qquad \text{ on } \Sigma_T,\\
			z(T,\cdot)=0   \qquad \text{ in } (0,1),
		\end{cases}
	\end{equation}
	for arbitrary $m\in  L_{\mathcal{F}}^2(0,T;L^2(0,1))$. By duality,
	\begin{multline}\label{equ5.24}
		\ds  -\E\int_{(0,1)}y_{0}z(0) dx=-\E\int_{Q_T} y_{\varepsilon} m dx dt+
		\E\int_{Q_T}F z dx dt\\ +\E\int_{0}^{T}\int_{\mathscr{O}} h_{\varepsilon} z dx dt +
		\E\int_{Q_T} GZ dx dt+\E\int_{Q_T}H_{\varepsilon}Z dx dt.
	\end{multline}
	Since $z\in L_{\mathcal{F}}^2(\Omega;C([0,T];L^2(0,1)))\cap L_{\mathcal{F}}^2(0,T;H_a^1)$ and $Z\in L_{\mathcal{F}}^2(0,T;L^2(0,1))$, each term in \eqref{equ5.24} is continuous with respect to the weak convergences in \eqref{equ5.22}. Passing to the limit as $\varepsilon\to 0$, we get
	\begin{multline}\label{equ5.25}
		\ds  -\E\int_{(0,1)}y_{0}z(0) dx=-\E\int_{Q_T} \widehat{y} m dx dt+
		\E\int_{Q_T}F z dx dt\\ +\E\int_{0}^{T}\int_{\mathscr{O}} \widehat{h} z dx dt +
		\E\int_{Q_T} GZ dx dt+\E\int_{Q_T}\widehat{H}Z dx dt.
	\end{multline}
	By the standard transposition argument, identity \eqref{equ5.25} exactly means that $(\widehat{y},\widehat{h},\widehat{H})$ is the solution of \eqref{form6.1}. On the other hand, the boundedness of $(y_{\varepsilon})$ in $L_{\mathcal{F}}^2(\Omega;C([0,T];L^2(0,1)))$ implies, up to extraction, that $y_{\varepsilon}(T)\rightharpoonup \widehat{y}(T)$ weakly in $L^2(\Omega;L^2(0,1))$. Since we already know that $y_{\varepsilon}(T)\to 0$ strongly in the same space, the weak limit is unique and therefore $\widehat{y}(T)=0$. Finally, estimate \eqref{equtheo5.2} follows from \eqref{nequa5.21} by weak lower semicontinuity.
\end{proof}

\section{Conclusion}
In this work, we established null controllability results for stochastic degenerate parabolic equations under weak and strong degeneracy assumptions. The analysis rests on two complementary Carleman arguments.

First, we derived a global Carleman estimate for a forward stochastic degenerate equation with multiplicative noise. This estimate yields an observability inequality and a unique continuation property, which in turn imply null controllability for the backward stochastic degenerate system.

Second, by adapting HUM to the degenerate stochastic setting and by proving a Carleman estimate for the backward equation with a weight that does not vanish at $t=0$, we obtained null controllability for a forward stochastic degenerate equation with two controls.

These results extend the available controllability theory beyond the prototype diffusion $a(x)=x^\alpha$ and provide a framework for further developments, including semilinear models and robustness questions for stochastic degenerate dynamics.

\section{Appendix}
In this section, we establish a Caccioppoli type inequality for a forward stochastic parabolic equation. Let $\mathscr{O}^{\prime}$ be an open subset of $(0,1)$ such that $\overline{\mathscr{O}^{\prime}}\subset (0,1)$, and let $\mathscr{O}^{\prime\prime}\Subset \mathscr{O}^{\prime}$ be a nonempty open subset. We consider
\begin{equation}\label{equ:5.1}
	d_t v(t,x)=\left(a(x)v_x\right)_xdt+bv(t,x)dt+f(t,x)dt+g(t,x)dW(t)
	\quad\text{ in } [0,T]\times\mathscr{O}^{\prime}.
\end{equation}

\begin{lemma}[Caccioppoli's inequality]\label{lem:caccioppoli}~\par 
	There exists a positive constant $C$ such that every solution $v$ of \eqref{equ:5.1} satisfies
	\begin{multline}\label{Caccioequa}
		\E\int_0^T\int_{\mathscr{O}^{\prime\prime}} v_x^2 e^{-2\mu\varphi}dxdt\\
		\leqslant C\left(\E\int_0^T\int_{\mathscr{O}^{\prime}} \mu^2\theta^2v^2  e^{-2\mu\varphi}dxdt
		+\E\int_0^T\int_{\mathscr{O}^{\prime}} f^2e^{-2\mu\varphi}dxdt
		+\E\int_0^T\int_{\mathscr{O}^{\prime}} g^2e^{-2\mu\varphi}dxdt\right),
	\end{multline}
	for all $\mu$ large enough.
\end{lemma} 	
\begin{proof}
	Let $\eta:\R\to\R$ be a smooth cut-off function satisfying $0\leqslant\eta\leqslant1$, $\eta=1$ in $\mathscr{O}^{\prime\prime}$,  $\mathrm{supp}\,\eta\subset\mathscr{O}^{\prime}$, and $\eta_x^2/\eta\in L^{\infty}(\mathscr{O}^{\prime})$. Applying It\^o's formula to
	$$\int_{\mathscr{O}^{\prime}}\eta v^2e^{-2\mu\varphi}dx,$$
	integrating over $[0,T]$ and taking expectation, we obtain
	\begin{align}
		&\E\int_{\mathscr{O}^{\prime}}\eta v^2(T)e^{-2\mu\varphi(T)}dx
		-\E\int_{\mathscr{O}^{\prime}}\eta v^2(0)e^{-2\mu\varphi(0)}dx \notag\\
		&=-2\E\int_0^T\int_{\mathscr{O}^{\prime}}\eta_xav_xv e^{-2\mu\varphi}dxdt
		-2\E\int_0^T\int_{\mathscr{O}^{\prime}}\eta av_x^2e^{-2\mu\varphi}dxdt \notag\\
		 &\quad+4\E\int_0^T\int_{\mathscr{O}^{\prime}}\eta\mu\varphi_xav_xv e^{-2\mu\varphi}dxdt
		+2\E\int_0^T\int_{\mathscr{O}^{\prime}}b\eta v^2e^{-2\mu\varphi}dxdt \notag\\
		 &\quad+2\E\int_0^T\int_{\mathscr{O}^{\prime}}\eta fve^{-2\mu\varphi}dxdt
		-2\E\int_0^T\int_{\mathscr{O}^{\prime}}\mu\eta\varphi_t v^2e^{-2\mu\varphi}dxdt
		+\E\int_0^T\int_{\mathscr{O}^{\prime}}\eta g^2e^{-2\mu\varphi}dxdt . \label{caccio-ito}
	\end{align}
	In the application used above, the Carleman weight vanishes at the temporal
	endpoints, so the two trace terms are obtained by a standard truncation and
	limit argument and do not contribute. Equivalently, for a regularized weight
	the nonnegative terminal trace is kept on the left and then discarded; the
	initial trace is handled before passing to the singular-weight limit.
	Thus,
	\begin{align*}
		\E\int_0^T\int_{\mathscr{O}^{\prime\prime}} v_x^2e^{-2\mu\varphi}dxdt
		 &\leqslant C\E\int_0^T\int_{\mathscr{O}^{\prime}}\eta av_x^2e^{-2\mu\varphi}dxdt\\
		 &\leqslant C\left(J_1+J_2+J_3+J_4+J_5\right),
	\end{align*}
	where
	\begin{align*}
		J_1 &=\E\int_0^T\int_{\mathscr{O}^{\prime}}|\eta_x|a|v_x||v|e^{-2\mu\varphi}dxdt,\\
		J_2 &=\E\int_0^T\int_{\mathscr{O}^{\prime}}\eta\mu|\varphi_x|a|v_x||v|e^{-2\mu\varphi}dxdt,\\
		J_3 &=\E\int_0^T\int_{\mathscr{O}^{\prime}}\eta|f||v|e^{-2\mu\varphi}dxdt,\\
		J_4 &=\E\int_0^T\int_{\mathscr{O}^{\prime}}\eta\left(|b|+\mu|\varphi_t|\right)v^2e^{-2\mu\varphi}dxdt,\\
		J_5 &=\E\int_0^T\int_{\mathscr{O}^{\prime}}\eta g^2e^{-2\mu\varphi}dxdt.
	\end{align*}
	For every $\varepsilon >0$, Young's inequality gives
	\begin{align*}
		J_1 &\leqslant \varepsilon\E\int_0^T\int_{\mathscr{O}^{\prime}}\eta av_x^2e^{-2\mu\varphi}dxdt
		+C_{\varepsilon}\E\int_0^T\int_{\mathscr{O}^{\prime}}v^2e^{-2\mu\varphi}dxdt,\\
		J_2 &\leqslant \varepsilon\E\int_0^T\int_{\mathscr{O}^{\prime}}\eta av_x^2e^{-2\mu\varphi}dxdt
		+C_{\varepsilon}\E\int_0^T\int_{\mathscr{O}^{\prime}}\mu^2\theta^2v^2e^{-2\mu\varphi}dxdt,\\
		J_3 &\leqslant C\E\int_0^T\int_{\mathscr{O}^{\prime}}f^2e^{-2\mu\varphi}dxdt
		+C\E\int_0^T\int_{\mathscr{O}^{\prime}}v^2e^{-2\mu\varphi}dxdt,\\
		J_4 &\leqslant C\E\int_0^T\int_{\mathscr{O}^{\prime}}\mu^2\theta^2v^2e^{-2\mu\varphi}dxdt.
	\end{align*}
	Here we used that $a$ is bounded from above and below on $\mathscr{O}^{\prime}$, and that $|\varphi_x|\leqslant C\theta$, $|\varphi_t|\leqslant C\theta^2$ on $[0,T]\times\mathscr{O}^{\prime}$. Choosing $\varepsilon >0$ small enough and taking $\mu$ large enough yields \eqref{Caccioequa}.
\end{proof}


\begin{thebibliography}{1}
	\bibitem{hjjaj} E. M. Ait Ben Hassi, F. Ammar Khodja, A. Hajjaj and L. Maniar, \textit{Carleman estimates and null controllability of coupled degenerate systems.} Evol. Equ. Control Theory \textbf{2}(2013), 441--459.
	\bibitem{barbuRascanu} V. Barbu, A. R\u{a}scanu, and G. Tessitore, \textit{Carleman estimate and controllability of linear stochastic heat equations.} Appl. Math. Optim., \textbf{47}(2003), pp. 97--120.
	\bibitem{camp} M. Campiti, G. Metafune, and D. Pallara, \textit{Degenerate self-adjoint evolution equations on the unit interval}, Semigroup Forum, \textbf{57} (1998), pp. 1--36.
	\bibitem{Fadi1} M. Fadili, \textit{Controllability of a backward stochastic cascade system of coupled parabolic heat equations by one control force.} Evol. Equ. Control Theory, doi:10.3934/eect.2022037.
	\bibitem{PratoZab} G. Da Prato and J. Zabczyk, \textit{Stochastic equations in infinite dimensions.} Cambridge University Press, Cambridge, 1992.
	\bibitem{Santamaria} V. Hern\'andez-Santamar\'ia, K. L. Balc'h and L. Peralta, \textit{Global null-controllability for stochastic semilinear parabolic equations.} preprint, arXiv: 2010.08854v1 (2020). 	
	\bibitem{hupeng} Y. Hu and S. Peng, \textit{Maximum principle for semilinear stochastic evolution control systems.} Stochastics Stochastic Rep. \textbf{33}(1990), 159--180.
	\bibitem{J.Lions1} J.-L. Lions, \textit{Optimal control of systems governed by partial differential equations.} Band 170. Springer-Verlag, New York-Berlin, 1971.
	\bibitem{liuyu} X. Liu and Y. Yu, \textit{Carleman estimates of some stochastic degenerate parabolic equations and application.} SIAM J. Control Optim \textbf{57}, \textbf{5}(2019) 3527--3552.
	\bibitem{Parpeng} E. Pardoux and S. Peng, \textit{Adapted solution of backward stochastic equations.} System Control Lett \textbf{14}(1990), 55--61. 
	\bibitem{ParRasca} E. Pardoux and A. R\u{a}scanu, \textit{Backward stochastic variational inequalities.} Stochastics Stochastic Rep. \textbf{67}(1999), 159--167.
	\bibitem{tangzhang} S. Tang and X. Zhang, \textit{Null controllability for forward and backward stochastic parabolic equations.} SIAM J. Control Optim. \textbf{48}(2009), 2191--2216.
	\bibitem{Tess} G. Tessitore, \textit{Existence, uniqueness and space regularity of the adapted solutions of a backward SPDE.} Stochastic Anal Appl. \textbf{14}(1996), 461--486. 
	\bibitem{Zab} J. Zabczyk, \textit{Mathematical Control Theory.} Birkh\"auser, Boston, 1995.
\end{thebibliography}
\end{document}